\documentclass[sn-aps, Numbered]{sn-jnl}

\usepackage{graphicx}%
\usepackage{multirow}%
\usepackage{amsmath,amssymb,amsfonts}%
\usepackage{amsthm}%
\usepackage{mathrsfs}%
\usepackage[title]{appendix}%
\usepackage{xcolor}%
\usepackage{textcomp}%
\usepackage{manyfoot}%
\usepackage{booktabs}%
\usepackage{algorithm}%
\usepackage{algorithmicx}%
\usepackage{algpseudocode}%
\usepackage{listings}%
\usepackage{subcaption}    
\usepackage{enumitem}
\usepackage{graphicx}
\usepackage{subcaption}
\usepackage{amsthm}

\newtheorem{assumption}{Assumption}
\theoremstyle{plain}%
\newtheorem{theorem}{Theorem}
\newtheorem{corollary}{Corollary}
\newtheorem{lemma}{Lemma}
\newtheorem{proposition}{Proposition}%
\newtheorem{remark}{Remark}%
\theoremstyle{thmstyletwo}%

\theoremstyle{thmstylethree}%
\newtheorem{definition}{Definition}%

\begin{document}

\title[Article Title]{A Time-Varying Coefficient-Based Fixed-Time Neurodynamic Approach for Generalized Monotone Inclusion Problems and Its Applications}

\author[1]{\fnm{Vajahat Karim}\sur{Khan}}\email{gh7904@myamu.ac.in}
\affil[1]{\orgdiv{Department of Mathematics}, \orgname{Aligarh Muslim University}, \city{Aligarh}, \postcode{202002}, \state{U.P.}, \country{India}}

\abstract{
This manuscript develops time-varying coefficients neurodynamic models (TVCNDMs) to solve the inclusion problem (IPs) $0 \in \mathcal{P}(w) + \mathcal{Q}(w)$, where $\mathcal{P}$ is a maximal set-valued and $\mathcal{Q}$ is a single-valued operator. The underlying operators are assumed to fulfill a generalized monotonicity condition that is less restrictive than the classical monotonicity assumption. Within this framework, existence and uniqueness of the model trajectories are established, and fixed-time ($\mathrm{FxT}$) convergence of the proposed TVCNDM to the solution of the IPs, along with explicit estimates of the settling time. Furthermore, the robustness of the proposed TVCNDM is investigated in the presence of bounded external disturbances, demonstrating that the $\mathrm{FxT}$ convergence. Numerical simulations support the theoretical results, illustrating rapid error decay and improved transient behavior achieved through the time-varying design.
}


%



\maketitle

\section{Introduction}\label{ICFxTsec1}
In this work, we investigate inclusion problem (IP) of finding ${\overline{\mathfrak{w}}}^*\in \mathscr{H}$ such that
\begin{equation}\label{inc1}
0 \in \mathcal{P}({\overline{\mathfrak{w}}}^*) + \mathcal{Q}({\overline{\mathfrak{w}}}^*),
\end{equation}
where $\mathcal{P} : \mathscr{H} \to 2^{\mathscr{H}}$ is a set-valued and 
$\mathcal{Q} : \mathscr{H} \to \mathscr{H}$ is a single-valued operator.

The framework \eqref{inc1} provides a powerful and versatile modeling tool that captures constrained optimization problems, variational inequalities, equilibrium models, Nash equilibrium, and fixed-point formulations; see \cite{8,CY,16, ICd2}.

A classical framework for solving \eqref{inc1} is provided by the forward-backward splitting method and its extensions \cite{37,42,19,plosevk,39,49}. While these methods are well understood in discrete settings, continuous-time dynamical models, often referred to as neurodynamic approaches, have recently attracted considerable attention due to their low computational complexity and strong stability properties \cite{1,9,11,13,21,26}. Most existing works, however, primarily establish asymptotic or exponential convergence. In many practical scenarios, such as robotics, networked control, and signal processing, guaranteeing convergence within a finite and predictable time horizon is crucial. This requirement motivates the study of finite-time and fixed-time ($\mathrm{FxT}$) stability. Finite-time stability ensures convergence in finite time but typically depends on initial conditions. $\mathrm{FxT}$ stability, introduced by Polyakov \cite{43}, addresses this limitation by ensuring a uniform upper bound on the settling time that does not depend on the initial conditions.

Although $\mathrm{FxT}$ convergence has attracted considerable attention in optimization and control \cite{22,23,32,33,FTC}, existing proximal dynamical models and recent continuous-time optimization methods with $\mathrm{FxT}$ convergence are predominantly developed under constant-gain structures and classical monotonicity assumptions \cite{ROCK,NHVF,TVl1}.

This work develops a novel $\mathrm{FxT}$ neurodynamic scheme incorporating time-varying coefficients to solve the IP \eqref{inc1}. The proposed dynamical model employs continuously varying gains to flexibly regulate convergence rates along trajectories. Working within a generalized monotonicity framework that allows the monotonicity modulus to take negative values, we demonstrate that the model admits unique trajectories and prove that the equilibrium point corresponding to the solution is $\mathrm{FxT}$ stable. The derived explicit settling-time estimates ensure convergence within a fixed time, irrespective of the initial conditions.

Furthermore, we take time-varying coefficient neurodynamic model (TVCNDM) exhibits strong robustness against bounded external disturbances. The $\mathrm{FxT}$ convergence property is shown to be preserved under interference, ensuring reliable and stable performance. Numerical simulations are presented to validate theoretical developments, where error convergence responses confirm rapid decay of the state error and highlight improved transient behavior achieved through the time-varying design.

The key contributions of this paper are outlined as follows:
\begin{enumerate}
    \item A FxT neurodynamic model with time-varying coefficients for solving IP.
    \item Existence, uniqueness, and FxT convergence established under generalized monotonicity assumptions.
    \item Explicit settling-time bounds revealing controllability of convergence speed via time-varying gains.
    \item Robustness analysis under bounded disturbances showing preservation of FxT stability.
   
    \item Numerical simulations are conducted to validate the theoretical results, demonstrating rapid error decay and improved transient performance achieved by the time-varying design.
\end{enumerate}

The remainder of the paper is organized as follows. Section~\ref{ICFxTsec2} presents preliminary concepts and technical tools. Section~\ref{ICFxTsec3} introduces the proposed neurodynamic models and establishes FxT convergence. Section~\ref{ICFxTsec4} investigates robustness under disturbances. Further, numerical simulations illustrating the effectiveness of the approach provides in section~\ref{ICFxTsec5}. Finally, Section~\ref{ICFxTsec6} concludes the paper.

\section{Preliminary}\label{ICFxTsec2}
Suppose that $\mathscr{H}$ is a finite-dimensional real Hilbert space with inner product $\langle \cdot,\cdot \rangle$ and norm $\|\cdot\|$ induced by the inner product. Let a proper, convex, and lower semicontinuous (l.s.c.) function 
$\Gamma:\mathscr{H}\to\mathbb{R}\cup\{+\infty\}$ and if the set
\[
\partial\phi(w)
=
\Big\{
u\in \mathscr{H}:\;
\Gamma(v)\ge\phi(w)+\langle u,v-w\rangle,
\ \forall v\in \mathscr{H}
\Big\}
\]
is nonempty. Then the set $\partial\phi(w)$ is siad to be  
\emph{subdifferential} of $\Gamma$ at $w\in\mathscr{H}$, and any element 
$u\in\partial\phi(w)$ is referred to as a \emph{subgradient} of $\Gamma$ at $w$.

Suppose $\Gamma$ is convex, l.s.c., and defined on all of $\mathscr{H}$. Then $\Gamma$ is continuous on $\mathscr{H}$ 
\cite[Corollary~8.30]{BauschkeCombettes}, and consequently subdifferentiable everywhere on $\mathscr{H}$ 
\cite[Proposition~16.14]{BauschkeCombettes}. 

Furthermore, let $\Upsilon:\mathscr{H}\to\mathbb{R}\cup\{+\infty\}$ be another l.s.c. function, which is proper and convex. Also, satisfying 
$\mathrm{dom}\,\Upsilon\cap\mathrm{int\,dom}\,\Gamma\neq\emptyset$ or $\mathrm{dom}\,\Gamma=\mathscr{H}$, where 
$\mathrm{dom}\,\Upsilon=\{w\in\mathscr{H}:\Upsilon(w)<+\infty\}$. 
Then the subdifferential sum rule holds:
\(
\partial(\Upsilon+\Gamma)=\partial \Upsilon+\partial\phi,
\)
see \cite{BauschkeCombettes}.

Let $\mathcal{C}\subset \mathscr{H}$ be a nonempty, closed, and convex set. 
The normal cone to $\mathcal{C}$ at a point $w\in \mathcal{C}$ is defined as
\[
\mathcal{N}_{\mathcal{C}}(w)
=
\Big\{
v\in \mathscr{H}:\;
\langle v, z-w\rangle \le 0,
\ \forall z\in \mathcal{C}
\Big\},
\]
while $\mathcal{N}_{\mathcal{C}}(w)=\emptyset$ whenever $w\notin \mathcal{C}$. If we take
\[
i_{\mathcal{C}}(w)=
\begin{cases}
0, & w\in \mathcal{C},\\
+\infty, & w\notin \mathcal{C}.
\end{cases}
\]
Then, it is callled indicator function of $\mathcal{C}$. It is well known that the subdifferential of the indicator function coincides with the normal cone, that is,
\(
\partial i_{\mathcal{C}}(w)=\mathcal{N}_{\mathcal{C}}(w), \ \forall w\in \mathscr{H}.
\)
For any $w\in \mathscr{H}$, the metric projection of $w$ onto $\mathcal{C}$ is defined by
\[
\mathcal{J}_{\mathcal{C}}(w)
=
\arg\min_{v\in \mathcal{C}}\|v-w\|.
\]
When $\mathcal{C}$ is closed and convex, the projection $\mathcal{J}_{\mathcal{C}}(w)$ is uniquely well-defined.

Let $\Upsilon:\mathscr{H}\to\mathbb{R}\cup\{+\infty\}$. Given $\mu>0$, the proximity operator of $\Upsilon$ with parameter $\mu$ is defined, for each $w\in\mathscr{H}$, as
\[
\mathrm{prox}_{\mu \Upsilon}(w)
=
\arg\min_{v\in \mathscr{H}}
\left\{
\Upsilon(v)+\frac{1}{2\mu}\|w-v\|^2
\right\}.
\]
Further details on these notions can be found in \cite{BauschkeCombettes}.

\subsection{Monotone Operators}
Let $\mathcal{P}:\mathscr{H}\to 2^{\mathscr{H}}$ be a set-valued operator. 
Its domain is defined by
\[
\mathrm{dom}\,\mathcal{P}
=
\{w\in\mathscr{H}:\mathcal{P}(w)\neq\emptyset\},
\]
and its graph is given by
\[
\mathrm{Gr}(\mathcal{P})
=
\{(w,v)\in \mathscr{H}\times \mathscr{H}: v\in \mathcal{P}(w)\}.
\]

\begin{definition}
The operator $\mathcal{P}:\mathscr{H}\to 2^{\mathscr{H}}$ is called 
\emph{$\Pi_{\mathcal{P}}$-monotone} if $\exists$ a constant 
$\Pi_{\mathcal{P}}\in\mathbb{R}$ s. t.
\[
\langle w-y,\,u-v\rangle
\ge
\Pi_{\mathcal{P}}\|w-y\|^2,
\quad
\forall w,y\in \mathscr{H},\ 
u\in \mathcal{P}(w),\ 
v\in \mathcal{P}(y).
\]
\end{definition}

\begin{definition}\cite{Minh}
A $\Pi_{\mathcal{P}}$-monotone operator $\mathcal{P}$ is maximal if no proper $\Pi_{\mathcal{P}}$-monotone extension of $\mathcal{P}$ exists.\\
\end{definition}

\begin{remark}
The parameter $\Pi_{\mathcal{P}}$ may be negative. 
The case $\Pi_{\mathcal{P}}=0$ corresponds to classical monotonicity, 
$\Pi_{\mathcal{P}}>0$ corresponds to strong monotonicity, 
and $\Pi_{\mathcal{P}}<0$ describes weak monotonicity.
\end{remark}

The subdifferential $\partial \Upsilon$ of a proper, convex, and lower semicontinuous function $\Upsilon:\mathscr{H}\to(-\infty,+\infty]$, with $\Upsilon \not\equiv +\infty$, is a classical example of a maximal monotone operator \cite{ROCK}.\\

\begin{definition}
We say that $\mathcal{Q}:\mathscr{H}\to\mathscr{H}$ is
\begin{itemize}
    \item Lipschitz continuous with $ \xi\ge 0$ if
    \[
    \|\mathcal{Q}(w_1)-\mathcal{Q}(w_2)\|\le  \xi\|w_1-w_2\|,
    \qquad \forall w_1,w_2\in\mathscr{H};
    \]
    \item \emph{$\zeta$-cocoercive} if
    \[
    \langle \mathcal{Q}(w_1)-\mathcal{Q}(w_2),\,w_1-w_2\rangle
    \ge
    \zeta\|\mathcal{Q}(w_1)-\mathcal{Q}(w_2)\|^2,
    \qquad \forall w_1,w_2\in\mathscr{H}.
    \]
\end{itemize}
\end{definition}

By the Cauchy--Schwarz inequality, every $\zeta$-cocoercive operator is $\frac{1}{\zeta}$-Lipschitz continuous.

The resolvent of a set-valued operator $\mathcal{P}:\mathscr{H}\to 2^{\mathscr{H}}$ with parameter $\mu>0$ is defined as
\[
\mathscr{R}_{\mu \mathcal{P}}
=
(\mathcal{I}+\mu \mathcal{P})^{-1},
\]
where $\mathcal{I}$ denotes the identity operator.\\

\begin{lemma}\label{GENMON}\cite{Minh}
Let $\mathcal{P}:\mathscr{H}\to 2^{\mathscr{H}}$ be a $\Pi_{\mathcal{P}}$-monotone set-valued operator and $\mathcal{Q}:\mathscr{H}\to\mathscr{H}$ be a single-valued mapping. 
Assume that $\mu>0$ satisfies $1+\mu\Pi_{\mathcal{P}}>0$. Consequently, we have:
\begin{enumerate}[label=\textnormal{(\roman*)}]
    \item  $\mathscr{R}_{\mu\mathcal{P}}$ is single-valued.
    \item $\mathscr{R}_{\mu\mathcal{P}}$ is $(1+\mu\Pi_{\mathcal{P}})$-cocoercive.
    \item $\mathrm{dom}\,\mathscr{R}_{\mu\mathcal{P}}=\mathscr{H}$ $\iff$ $\mathcal{P}$ is maximal $\Pi_{\mathcal{P}}$-monotone.\\
\end{enumerate}
\end{lemma}

We denote by $\mathrm{Fix}(\mathcal{P})$ the set of fixed points of an operator $\mathcal{P}$ and by $\mathrm{zer}\,\mathcal{P}$ the set of its zeros.

\begin{lemma}\label{lem fixed point}\cite{NHVF} Suppose the operator $\mathcal{P}:\mathscr{H}\to 2^{\mathscr{H}}$ is $\Pi_{\mathcal{P}}$-monotone and $\mathcal{Q}:\mathscr{H}\to\mathscr{H}$ be single-valued. 
If $\mu>0$ satisfies $1+\mu\Pi_{\mathcal{P}}>0$, then a point ${\overline{\mathfrak{w}}}^*\in\mathscr{H}$ solves the IP \eqref{inc1} $\iff$
\[
\mathscr{R}_{\mu\mathcal{P}}({\overline{\mathfrak{w}}}^*-\mu\mathcal{Q}({\overline{\mathfrak{w}}}^*))={\overline{\mathfrak{w}}}^*,
\]
that is,
\[
{\overline{\mathfrak{w}}}^*\in \mathrm{Fix}\!\left(\mathscr{R}_{\mu\mathcal{P}}(\mathcal{I}-\mu\mathcal{Q})\right).
\]
\end{lemma}
For further background on monotone operators and their applications in optimization and variational analysis, we refer the reader to \cite{AVR,BauschkeCombettes,Minh}. 
In what follows, we impose the standing assumption below.\\

\begin{assumption}\label{ICAsmmain}
The operator $\mathcal{P}$ is maximal $\Pi_{\mathcal{P}}$-monotone, and $\mathcal{Q}$ is $\Pi_{\mathcal{Q}}$-monotone as well as $\xi$-Lipschitz continuous. 
Moreover, $\mu>0$ is chosen s. t.
\[
1+\mu\Pi_{\mathcal{P}}>0,
\quad ; \quad
2(\Pi_{\mathcal{P}}+\Pi_{\mathcal{Q}})+\mu\Pi_{\mathcal{P}}^{2}-\mu  \xi^{2}>0.
\]
\end{assumption}
\begin{remark}
It follows from \cite{NHVF} that Assumption~\eqref{ICAsmmain} is valid under any of the following conditions:
\begin{enumerate}[label=\textnormal{(\roman*)}]
    \item $\Pi_{\mathcal{P}}+\Pi_{\mathcal{Q}}>0$.
    \item $\Pi_{\mathcal{P}}+\Pi_{\mathcal{Q}}=0 \ ; \ \Pi_{\mathcal{P}}^{2}> \xi^{2}$.
    \item $\Pi_{\mathcal{P}}+\Pi_{\mathcal{Q}}<0$ \ ; \ $\Pi_{\mathcal{P}}> \xi$.\\
\end{enumerate}
\end{remark}

It is worth noting that Assumption~\eqref{ICAsmmain} does not necessarily imply (strong) monotonicity of the operator $\mathcal{P}+\mathcal{Q}$. 
Specifically, if $\mathcal{P}$ is maximal monotone (so that $\Pi_{\mathcal{P}}=0$) and $\mathcal{Q}$ is strongly monotone (that is, $\Pi_{\mathcal{Q}}>0$), then Assumption~\eqref{ICAsmmain} is fulfilled for all $\mu\in(0,2\Pi_{\mathcal{Q}}/ \xi^{2})$.\\

\begin{lemma} \cite{NHVF,PDINC}\label{tr1}
Let $\mathcal{P}:\mathscr{H}\to2^{\mathscr{H}}$ be a $\Pi_{\mathcal{P}}$-monotone set-valued operator and $\mathcal{Q}:\mathscr{H}\to\mathscr{H}$ be a single-valued mapping. Suppose that Assumption~\eqref{ICAsmmain} holds and define
\[
\Psi(w):=\mathscr{R}_{\mu\mathcal{P}}(w-\mu\mathcal{Q}(w)),
\qquad \mu>0.
\]
Then, $\exists \ \aleph\in(0,1)$, the following holds for every
\(w\in\mathscr{H}\setminus\{\overline{\mathfrak{w}}^{*}\}\):
\begin{enumerate}[label=\textnormal{(\roman*)}]
    \item \label{patr1} The operator $\mathscr{R}_{\mu \mathcal{P}}\circ (\mathcal{I}-\mu \mathcal{Q})$ is Lipschitz continuous with constant $\aleph$.
    
    \item \label{pbtr1} The fixed-point set $\mathrm{Fix}\!\left(\mathscr{R}_{\mu \mathcal{P}}\circ (\mathcal{I}-\mu \mathcal{Q})\right)$ is nonempty and singleton.
    
    \item \label{pctr1} $\|\Psi(w)- {\overline{\mathfrak{w}}}^*\|\leq \aleph\|w-{\overline{\mathfrak{w}}}^*\|, \ \forall \ {\overline{\mathfrak{w}}}^*\in \mbox{\rm Fix}(\mathscr{R}_{\mu \mathcal{P}} (\mathcal{I}-\mu \mathcal{Q}))$.
    
    \item \label{pdtr1} $(1- \aleph)\|w-{\overline{\mathfrak{w}}}^*\|^2\leq \langle w-\Psi(w) ,w-{\overline{\mathfrak{w}}}^* \rangle$.
    
    \item \label{petr1} $(1- \aleph)\|w-{\overline{\mathfrak{w}}}^*\|\leq \|w-\Psi(w)\|\leq (1+\aleph)\|w-{\overline{\mathfrak{w}}}^*\|$.
    
    \item \label{pftr1} $\langle w-\Psi(w) ,w-{\overline{\mathfrak{w}}}^* \rangle \leq (1+\aleph)\|w-{\overline{\mathfrak{w}}}^*\|^2$.\\ 
\end{enumerate}
\end{lemma}

\begin{theorem}\label{rem2}
Suppose $\mathcal{P}:\mathscr{H}\to2^{\mathscr{H}}$ be a $\Pi_{\mathcal{P}}$-monotone set-valued operator and 
$\mathcal{Q}:\mathscr{H}\to\mathscr{H}$ be a single-valued mapping. 
Assume that $\mu>0$ satisfies $1+\mu\Pi_{\mathcal{P}}>0$ and that Assumption~\eqref{ICAsmmain} hold. 
Then IP \eqref{inc1} admits a unique solution.
\end{theorem}

\begin{proof}
From Part~\eqref{pbtr1} of Lemma~\ref{tr1}, the operator 
$\mathscr{R}_{\mu\mathcal{P}}\circ(\mathcal{I}-\mu\mathcal{Q})$
possesses a unique fixed point. 
By Lemma~\ref{lem fixed point}, this fixed point is equivalent to the unique solution of the IP~\eqref{inc1}. 
\end{proof}

\begin{remark}
Observe that Assumption~\eqref{ICAsmmain} is satisfied whenever the parameters 
$\Pi_{\mathcal{P}}, \Pi_{\mathcal{Q}}, L$, and $\mu$ fulfill one of the following conditions:
\begin{enumerate}[label=\textnormal{(\roman*)}]
    \item $\Pi_{\mathcal{P}}+\Pi_{\mathcal{Q}}>0$.
    \item $\Pi_{\mathcal{P}}+\Pi_{\mathcal{Q}}=0$ and $\Pi_{\mathcal{P}}^{2}> \xi^{2}$.
    \item $\Pi_{\mathcal{P}}+\Pi_{\mathcal{Q}}<0$ and $\Pi_{\mathcal{P}}> \xi$.
\end{enumerate}
Consequently, under any of these conditions, the solution of the problem \eqref{inc1} exists and is uniquely determined.
\end{remark}

\subsection{Conditions for Lyapunov Stability}
The proposed TVCNDM is formulated as follows:
\begin{equation}\label{eq:system}
\dot{w} =  \Sigma(s,w(s)), \qquad w(s_0)=w_0 ,
\end{equation}
where $w:[s_0,+\infty)\to \mathscr{H}$, and 
$ \Sigma:[s_0,+\infty)\times\mathscr{H} \to \mathscr{H}$ 
is continuous.\\

\begin{definition}\cite{ICd1,ICd3}\label{EPdef1}
Let ${\overline{\mathfrak{w}}}^*\in\mathscr{H}$ is an equilibrium point of
	\eqref{eq:system} if
	\begin{equation}
		\label{eq:equilibrium}
		 \Sigma(s,{\overline{\mathfrak{w}}}^*) = 0, \quad \forall \, s > s_0.
	\end{equation}. Then  model \eqref{eq:system} is said to be
\begin{enumerate}
	\item \texttt{Lyapunov stability:}  
	For every $\mathcal{E}>0$, $\exists$ $\chi=\chi(\mathcal{E},s_0)$ satisfying
	\begin{equation}
		\label{eq:lyapunov_stable}
		\|w_0-{\overline{\mathfrak{w}}}^*\|<\chi\;\Rightarrow\;
		\|w(s)-{\overline{\mathfrak{w}}}^*\|<\mathcal{E}, \quad \forall\, s>s_0.
	\end{equation}
	\item \texttt{FT convergence}, if  $\exists$ settling-time function $\mathcal{T}:\mathscr{H}\setminus\{\overline{\mathfrak{w}}^{*}\}\to(0,\infty) \ (\text{i.e.}\ \mathcal{T} < \infty)$ satisfying $\forall\, w(0)\in\mathscr{H}\setminus\{\overline{\mathfrak{w}}^{*}\}$
	\[
	w(s)=\overline{\mathfrak{w}}^{*},
	\qquad
	\forall\, s\ge \mathcal{T}(w(0)),
	\]
	where $\mathcal{T}(w(0))<\infty$.
	\item \texttt{\rm{FxT} stable}, if it is FT stable with $\exists$ an upper bound $\mathcal{T}_{\max} < \infty$ satisfying
	\[
	\sup_{w_0 \in \mathscr{H}} \mathcal{T}(w_0) \leq \mathcal{T}_{\max}.
	\]
\end{enumerate}

\end{definition}

\begin{lemma}\cite{TVl1}\label{lem2}
Suppose that $\exists$ a continuously differentiable function
$\mathcal{L}:[s_0,+\infty)\times\mathscr{H} \to \mathbb{R}_+$, which positive definite
satisfying
\begin{equation}
\dot{\mathcal{L}}(s,w(s)) \le
-m_1 \phi_1(s) {\mathcal{L}}^{\vartheta_1}(s,w(s))
-m_2 \phi_2(s) {\mathcal{L}}^{\vartheta_2}(s,w(s)),
\label{eq_lemma_condition}
\end{equation}
where $m_1>0$, $m_2>0$, $0<\vartheta_1<1$, $\vartheta_2>1$ and 
\(
\dot{\mathcal{L}}(s,w) :=
\frac{\partial \mathcal{L}}{\partial s}(s,w)
+
\frac{\partial \mathcal{L}}{\partial w}(s,w)\dot w.
\)
Suppose $\phi_1,\phi_2:[s_0,+\infty)\to \mathbb{R}_{++}$ be continuous.
Define
\[
\mathcal{F}_1(s) = \int_{s_0}^{s} \phi_1(\varkappa)\,d{\varkappa},
\qquad
\mathcal{F}_2(s) = \int_{s_0}^{s} \phi_2(\varkappa)\,d{\varkappa}.
\]
If
\(
\frac{1}{m_2(\vartheta_2-1)}
<
\int_{s_0}^{+\infty} \phi_2(\varkappa)\,d{\varkappa}
\)
and
\(
\mathcal{F}_1\!\circ\!\mathcal{F}_2^{-1}
\!\left(
\frac{1}{m_2(\vartheta_2-1)}
\right)
+
\frac{1}{m_1(1-\vartheta_1)}
<
\int_{s_0}^{+\infty} \phi_1(\varkappa)\,d{\varkappa},
\)
then $\exists$ $\mathcal{T}>s_0$ such that $\mathcal{L}(s,w(s))\equiv0, \ \forall \ s\ge \mathcal{T}$.
Moreover,
\begin{equation}
\mathcal{T} \le
\mathcal{F}_1^{-1}
\left(
\mathcal{F}_1\!\circ\!\mathcal{F}_2^{-1}
\!\left(
\frac{1}{m_2(\vartheta_2-1)}
\right)
+
\frac{1}{m_1(1-\vartheta_1)}
\right).
\label{eq_T_bound_lemma}
\end{equation}
\end{lemma}

\section{A Nominal Time-varying Neurodynamic Model for IPs}\label{ICFxTsec3}
To solve IP \eqref{inc1}, we take nominal TVCNDM as follows:
\begin{equation}\label{hdl2}
\dot{w}(s)
=
-\sigma(s)\Big(w(s) - \mathscr{R}_{\mu \mathcal{P}}\big(w(s) - \mu \mathcal{Q}(w(s))\big)\Big),
\end{equation}
where $\sigma(s)>0$ be a continuous time-dependent $\forall s>0$. If we  define
\(
\Psi(w):=\mathscr{R}_{\mu \mathcal{P}}(w-\mu \mathcal{Q}(w)).
\)
Then model \eqref{hdl2} can be rewritten as
\begin{equation}\label{1stNDS}
\dot{w}=-\sigma(s)\big(w-\Psi(w)\big).
\end{equation}

The next result reveals a fundamental relationship between the solution of \eqref{inc1} and the equilibrium points of \eqref{1stNDS}.\\

\begin{proposition}\label{prop3.1}
Suppose $\mathcal{P} : \mathscr{H} \to 2^{\mathscr{H}}$ be $\Pi_{\mathcal{P}}$-monotone and $\mathcal{Q} : \mathscr{H} \to \mathscr{H}$ be single-valued. If $\mu>0$ satisfies $1+\mu\Pi_{\mathcal{P}}>0$, then ${\overline{\mathfrak{w}}}^* \in \mathscr{H}$ is a solution of \eqref{inc1} $\iff$ it is an equilibrium point of \eqref{1stNDS}.
\end{proposition}

\begin{proof}
The assertion is a direct consequence of Lemma~\ref{lem fixed point}, where the solutions of \eqref{inc1} are characterized as the fixed points of $\Psi$, together with the definition of equilibrium of \eqref{1stNDS}.
\end{proof}

\subsection{TVCNDM for IPs}
The FxT TVCNDM, we formulate as:

	\begin{equation}\label{INFxTM}
		\dot{w}=-\theta(w,s)(w-\Psi(w)),
	\end{equation} 
	where 
	\begin{equation}\label{dnrho} 
		\theta(w,s):=\begin{cases} 
			\dfrac{ \phi_1(s)}{\|w-\Psi(w)\|^{1-\lambda_1}}& +\dfrac{ \phi_1(s)}{\|w-\Psi(w)\|^{1-\lambda_2}},\mbox{ if } w\in   \mathscr{H}\setminus {\rm Fix}(\Psi),\\
			0 & \mbox{otherwise}, 
		\end{cases}
	\end{equation} 
with $0<\lambda_1<1$, $\lambda_2>1$.
$\phi_1(s),\phi_2(s):[s_0,+\infty)\rightarrow\mathbb{R}_{++}$ are continuous
functions. Also, TVCNDM \eqref{INFxTM} can be written in the following form:
\begin{equation}
\dot{w} = -\phi_1(s)  \Sigma(w) \left[ \|w-\Psi(w)\|^{\lambda_1} \left( 1 + \frac{\phi_2(s)}{\phi_1(s)} \|w-\Psi(w)\|^{\lambda_2- \lambda_1} \right) \right], 
\end{equation}
where 
\[
 \Sigma(w) := \frac{w -\Psi(w)}{\|w -\Psi(w)\|}.
\]

\begin{remark}
\begin{enumerate}[label=\textnormal{(\roman*)}]
\item As we take  $\phi_2(s) = 1$, TVCNDM \eqref{INFxTM} is become
\[
\dot{w} = -\phi_1(s) \Sigma(w) \left[ \|w-\Psi(w)\|^{\lambda_1} \left( 1 + \frac{1}{\phi_1(s)} \|w-\Psi(w)\|^{\lambda_2- \lambda_1} \right) \right].
\]

\item If $b_2=0$ and $\phi_1(s)$ is a constant function, then the model \eqref{INFxTM} reduces to the nominal model studied in \cite{NH}, namely
\begin{equation}\label{reduemodel}
\dot{w}(s)
=
-\sigma\big(w(s)-\Psi(w)\big),
\end{equation}
where $\sigma>0$. 
Then model \eqref{reduemodel} was analyzed in \cite{NH} in the context of finite-time stability for generalized monotone IPs.
\end{enumerate}
\end{remark}

\begin{proposition}\label{prop3.5}
The models \eqref{1stNDS} and \eqref{INFxTM} share the same equilibrium points.
\end{proposition}

\begin{proof}
The claim follows immediately from \eqref{dnrho}, by the definition of $\theta(\cdot,s)$.
Indeed,
\[
\theta(w,s)=0
\quad \Longleftrightarrow \quad
w \in \mathrm{Fix}\!\left(\mathscr{R}_{\mu \mathcal{P}} \circ (\mathcal{I}-\mu \mathcal{Q})\right),
\]
which matches the equilibrium condition of model \eqref{1stNDS}.
\end{proof}

\begin{proposition}\label{prop3.6}
Suppose $\mathcal{P} : \mathscr{H} \to 2^{\mathscr{H}}$, $\mathcal{Q} : \mathscr{H} \to \mathscr{H}$, and $\mu>0$ satisfy Assumption~\eqref{ICAsmmain}. Then ${\overline{\mathfrak{w}}}^*$ solves the IP \eqref{inc1} $\iff$ it is an equilibrium point of \eqref{INFxTM}.
\end{proposition}

\begin{proof}
Combining Propositions~\ref{prop3.5} and~\ref{prop3.1} yields the desired conclusion.
\end{proof}

\subsection{FxT Convergence}
The well-posedness of the proposed TVCNDM is established in the following discussion.\\

\begin{proposition}\cite{Zheng}\label{unique sol da sy}
Consider the model \eqref{INFxTM}. Suppose that the mapping $\Psi: \mathscr{H}\to \mathscr{H}$ is locally Lipschitz continuous and satisfies
\[
\Psi({\overline{\mathfrak{w}}}^*)={\overline{\mathfrak{w}}}^*
\quad\text{and}\quad
\langle w-{\overline{\mathfrak{w}}}^*,\, w-\Psi(w)\rangle>0,
\quad
\forall w\in\mathscr{H}\setminus\{{\overline{\mathfrak{w}}}^*\},
\]
for some ${\overline{\mathfrak{w}}}^*\in \mathscr{H}$. Then the vector field of \eqref{INFxTM} is continuous on $\mathscr{H}$, and for any initial condition the differential equation admits a unique global solution defined for all $s\ge0$.\\
\end{proposition}

\begin{remark}\label{rem3}
Define
\(
\Psi(w):=\mathscr{R}_{\mu \mathcal{P}}\big(w-\mu \mathcal{Q}(w)\big),
\ w\in\mathscr{H}.
\)
By Lemma~\ref{tr1}, the operator $\Psi$ possesses a unique fixed point ${\overline{\mathfrak{w}}}^*$, and hence the model admits a unique equilibrium. Moreover, for any $w\neq {\overline{\mathfrak{w}}}^*$,
\begin{equation}\label{eq13}
\begin{aligned}
\langle w-{\overline{\mathfrak{w}}}^*, w-\Psi(w)\rangle
&=
\|w-{\overline{\mathfrak{w}}}^*\|^2
+
\left\langle w-{\overline{\mathfrak{w}}}^*,\,
{\overline{\mathfrak{w}}}^*-\mathscr{R}_{\mu \mathcal{P}}(w-\mu \mathcal{Q}(w))
\right\rangle .
\end{aligned}
\end{equation}
Applying the Cauchy--Schwarz inequality together with Lemma~\ref{tr1} yields
\begin{equation}\label{eq14}
\left\langle w-{\overline{\mathfrak{w}}}^*,\, 
{\overline{\mathfrak{w}}}^*-\mathscr{R}_{\mu \mathcal{P}}(w-\mu \mathcal{Q}(w))
\right\rangle
\ge
-\,\aleph\|w-{\overline{\mathfrak{w}}}^*\|^2,
\end{equation}
for some $\aleph\in(0,1)$. Combining \eqref{eq13} and \eqref{eq14}, we obtain
\[
\langle w-{\overline{\mathfrak{w}}}^*,\, w-\Psi(w)\rangle
\ge
(1-\aleph)\|w-{\overline{\mathfrak{w}}}^*\|^2
>
0,
\qquad
\forall w\in\mathscr{H}\setminus\{{\overline{\mathfrak{w}}}^*\}.
\]
\end{remark}

We now turn to the next principal result of the section. Let ${\overline{\mathfrak{w}}}^*\in\mathscr{H}$ be the unique equilibrium point of \eqref{INFxTM}  tp solve IP~\eqref{inc1}. Then it is FxT stable.

\begin{theorem}\label{tr2}
Consider an $\Pi_{\mathcal{P}}$-monotone set-valued mapping $\mathcal{P}: \mathscr{H} \to 2^\mathscr{H}$ and a single-valued operator $\mathcal{Q}: \mathscr{H} \to \mathscr{H} $, along with a positive parameter $\mu$ s. t.  assumption \eqref{ICAsmmain} is satisfied. Furthermore, Suppose that
\(
\frac{1}{m_2(\vartheta_2-1)}
<
\int_{s_0}^{+\infty} \phi_2(\varkappa)\,d{\varkappa}
\)
and
\(
{\mathcal{F}}_1\!\circ\!{\mathcal{F}}_2^{-1}
\left(
\frac{1}{m_2(\vartheta_2-1)}
\right)
+
\frac{1}{m_1(1-\vartheta_1)}
<
\int_{s_0}^{+\infty} \phi_1(\varkappa)\,d{\varkappa},
\)
where
\(
\vartheta_1=\frac{1+\lambda_1}{2},\
\vartheta_2=\frac{1+\lambda_2}{2},
\
m_1=
2^{\vartheta_1}
\frac{(1- \aleph)}{\mathcal (1+\aleph)^{1-\lambda_1}},
\) and \(
m_2=
2^{\vartheta_2}(1- \aleph)^{\lambda_2}
\)  for some $
\aleph\in(0,1)$ with
\[
{\mathcal{F}}_1(s)=\int_0^s \phi_1(\varkappa)\,d{\varkappa},
\qquad
{\mathcal{F}}_2(s)=\int_0^s \phi_2(\varkappa)\,d{\varkappa},
\]
then $w(s)$ of \eqref{INFxTM} converges to ${\overline{\mathfrak{w}}}^*$ with settling time in \eqref{eq_T_bound_lemma}.
\end{theorem}
\begin{proof}
The operator 
\(
\mathscr{R}_{\mu \mathcal{P}} \circ (\mathcal{I} - \mu \mathcal{Q})
\)  holds (i) condition of Lemma~\ref{tr1}, 
Moreover, by Theorem~\ref{rem2}, the IP admits a unique solution. In addition, Remark~\ref{rem3} ensures that the assumptions of Proposition~\ref{unique sol da sy} are satisfied. Consequently, the TVCNDM \eqref{INFxTM} admits a unique global solution for every initial condition and is well defined for all forward times.\\
\noindent
Consider the Lyapunov function candidate
\begin{equation}\label{THL1}
\mathcal{L}(s,w(s))=\frac{1}{2}\|w-{\overline{\mathfrak{w}}}^*\|^2.
\end{equation}
Then $\mathcal{L}(w)=0 \iff w={\overline{\mathfrak{w}}}^*$ and $\mathcal{L}(w)\to\infty$ as $\|w\|\to\infty$. Taking the time derivative of \eqref{THL1}, we obtain
\begin{align}
\notag\dot{\mathcal{L}}(s,w(s))
&=
\left\langle w-{\overline{\mathfrak{w}}}^*,\dot w\right\rangle
\\
&=
-
\left(
\phi_1(s)\frac{1}{\|w-\Psi(w)\|^{1-\lambda_1}}
+
\phi_2(s)\frac{1}{\|w-\Psi(w)\|^{1-\lambda_2}}
\right)
\cdot
\langle w-{\overline{\mathfrak{w}}}^*, w-\Psi(w)\rangle.
\end{align}
From Lemma \ref{tr1}, we have
\[
\langle w-{\overline{\mathfrak{w}}}^*, w-\Psi(w)\rangle
\ge
(1- \aleph) \|w-{\overline{\mathfrak{w}}}^*\|^2 \] and \[(1- \aleph)\|w-{\overline{\mathfrak{w}}}^*\|\leq\|w-\Psi(w)\|
\le
\mathcal (1+\aleph) \|w-{\overline{\mathfrak{w}}}^*\|.
\]
Thus,
\begin{align}\label{THL2}
\dot{\mathcal{L}}(s,w(s))
&\le
-
(1- \aleph) \phi_1(s)
\frac{\|w-{\overline{\mathfrak{w}}}^*\|^2}
{\mathcal (1+\aleph)^{1-\lambda_1}\|w-{\overline{\mathfrak{w}}}^*\|^{1-\lambda_1}}
-
(1- \aleph) \phi_2(s)
\frac{\|w-{\overline{\mathfrak{w}}}^*\|^2}
{(1- \aleph)^{1-\lambda_2}\|w-{\overline{\mathfrak{w}}}^*\|^{1-\lambda_2}}
\\
&=
-
\frac{(1- \aleph)}{\mathcal (1+\aleph)^{1-\lambda_1}}
\phi_1(s)
\|w-{\overline{\mathfrak{w}}}^*\|^{1+\lambda_1}
-
(1- \aleph)^{\lambda_2} \phi_2(s)
\|w-{\overline{\mathfrak{w}}}^*\|^{1+\lambda_2}.
\end{align}
Combining \eqref{THL1} and \eqref{THL2}, we obtain
\begin{align}\label{ThL3}
\dot{\mathcal{L}}(s,w(s))
&\le
-
2^{\frac{1+\lambda_1}{2}}
\frac{(1- \aleph)}{\mathcal (1+\aleph)^{1-\lambda_1}}
\phi_1(s)
{\mathcal{L}}^{\frac{1+\lambda_1}{2}}
-
2^{\frac{1+\lambda_2}{2}}
(1- \aleph)^{\lambda_2}
\phi_2(s)
{\mathcal{L}}^{\frac{1+\lambda_2}{2}}.
\end{align}
Let
\(
\vartheta_1=\frac{1+\lambda_1}{2},\quad
\vartheta_2=\frac{1+\lambda_2}{2},
\)
\(
m_1=
2^{\vartheta_1}
\frac{(1- \aleph)}{\mathcal (1+\aleph)^{1-\lambda_1}},
\) and \(
m_2=
2^{\vartheta_2}(1- \aleph)^{\lambda_2}.
\) Then inequality \eqref{ThL3}, becomes
\[
\dot{\mathcal{L}}
\le
-
m_1 \phi_1(s) {\mathcal{L}}^{\vartheta_1}
-
m_2 \phi_2(s) {\mathcal{L}}^{\vartheta_2}.
\]
Applying Lemma \ref{lem2} with define
\(
\mathcal{F}_1(s) = \int_{s_0}^{s} \phi_1(\varkappa)\,d{\varkappa},
\
\mathcal{F}_2(s) = \int_{s_0}^{s} \phi_2(\varkappa)\,d{\varkappa}
\). This establishes the FxT convergence of trajectory $w(s)$ to ${\overline{\mathfrak{w}}}^*$
, with the settling-time bound given by \eqref{eq_T_bound_lemma}.
\end{proof}
Next, we derive several settling-time bounds corresponding to different choices of the time-varying coefficient functions in \eqref{INFxTM}.

\begin{corollary}\label{IncCL1}
Suppose assumption $\eqref{ICAsmmain}$ hold. Let
\(
\vartheta_1=\frac{1+\lambda_1}{2},\quad
\vartheta_2=\frac{1+\lambda_2}{2},
\)
\(
m_1=
2^{\vartheta_1}
\frac{(1- \aleph)}{\mathcal (1+\aleph)^{1-\lambda_1}},
\) and \(
m_2=
2^{\vartheta_2}(1- \aleph)^{\lambda_2}.
\)
where $\varpi_0>0$, $ \varrho_0>0$, and $\tau>0$.

\begin{enumerate}[label=\textnormal{(\roman*)}]

\item If one selects $\phi_1(s)=\varpi_0$ and $\phi_2(s)= \varrho_0$ in TVCNDM~\eqref{INFxTM}, then, $w(s)$ reaches the equilibrium point ${\overline{\mathfrak{w}}}^*$ in FxT, with the settling-time bound given by
\begin{equation}\label{T1}
T \le 
\frac{1}{\varpi_0 m_1 (1-\vartheta_1)}
+
\frac{1}{ \varrho_0 m_2 (\vartheta_2-1)}.
\end{equation}

\item If one selects 
\(
\phi_1(s)=\frac{\varpi_0}{t+\tau} \ \text{and}\ 
\phi_2(s)=\frac{ \varrho_0}{t+\tau}
\)
in TVCNDM~\eqref{INFxTM}, then, $w(s)$ reaches the equilibrium point ${\overline{\mathfrak{w}}}^*$ in FxT, with the settling-time bound given by
\begin{equation}\label{T2}
T \le 
\tau \exp\!\left(
\frac{1}{\varpi_0 m_1 (1-\vartheta_1)}
+
\frac{1}{ \varrho_0 m_2 (\vartheta_2-1)}
\right)
-\tau.
\end{equation}

\item If one selects 
\(
\phi_1(s)=\varpi_0 
\ \text{and} \
\phi_2(s)=\frac{ \varrho_0}{t+\tau}
\)
in TVCNDM~\eqref{INFxTM}, then, $w(s)$ reaches the equilibrium point ${\overline{\mathfrak{w}}}^*$ in FxT, with the settling-time bound given by
\begin{equation}\label{T3}
T \le 
\tau \exp\!\left(
\frac{1}{ \varrho_0 m_2 (\vartheta_2-1)}
\right)
+
\frac{1}{\varpi_0 m_1 (1-\vartheta_1)}
-\tau.
\end{equation}

\item If one selects 
\(
\phi_1(s)=\frac{\varpi_0}{(s+\tau)^{\alpha_1}}
\ \text{and}
\ \phi_2(s)=\frac{ \varrho_0}{(s+\tau)^{\alpha_2}},
\)
where $\alpha_1 \neq 1$ and $\alpha_2 \neq 1$, s. t.
\[
\alpha=
\frac{1-\alpha_2}{ \varrho_0 m_2 (\vartheta_2-1)}
+
\tau^{1-\alpha_2}
>0,
\quad
\text{and}
\quad
\alpha^{\frac{1-\alpha_1}{1-\alpha_2}}
+
\frac{1-\alpha_1}{\varpi_0 m_1 (1-\vartheta_1)}
>0,
\]
then, $w(s)$ reaches the equilibrium point ${\overline{\mathfrak{w}}}^*$ in $\mathrm{FxT}$, with the settling-time bound given by
\begin{equation}\label{T4}
T \le 
\left(
\alpha^{\frac{1-\alpha_1}{1-\alpha_2}}
+
\frac{1-\alpha_1}{\varpi_0 m_1 (1-\vartheta_1)}
\right)^{\!\frac{1}{1-\alpha_1}}
-\tau.
\end{equation}

\end{enumerate}
\end{corollary}
By Corollary~\ref{IncCL1}, the settling-time bounds are functions of the fixed values $\lambda_1$, $\lambda_2$, $\Pi_{\mathcal{P}}$, $\Pi_{\mathcal{Q}}$, $\xi$, $\mu$, $\phi_1(s)$, and $\phi_2(s)$.

\section{Disturbance Robustness Analysis}\label{ICFxTsec4}
In this section, we show that the $\mathrm{FxT}$ convergence of TVCNDM~\eqref{INFxTM} is preserved in the presence of a bounded external disturbance $\mathcal{D}:\mathscr{H} \to \mathscr{H}$. 
\begin{equation}
\label{eqDistr}
\dot{w}
=
-
\left(
\frac{\phi_1(s)}{\|w-\Psi(w)\|^{1-\lambda_1}}
+
\frac{\phi_2(s)}{\|w-\Psi(w)\|^{1-\lambda_2}}
\right)
(w-\Psi(w))
+
\mathcal{D}(w),
\end{equation}

\begin{assumption}
\label{assINDIs}
For some constant $\delta>0$ and $w\ne\overline{\mathfrak{w}}^{*}$, we have
\[
\|\mathcal{D}(w)\|<\delta\|w-\overline{\mathfrak{w}}^{*}\|,
\quad \forall\, w\in\mathscr{H}.
\]
\end{assumption}

\begin{theorem} Under Assumption~\ref{ICAsmmain} and \ref{assINDIs} hold, Let
	
	\(\min\left\{
\frac{(1- \aleph)}{\mathcal (1+\aleph)^{1-\lambda_1}}\phi_1(s),
\;
(1- \aleph)^{\lambda_2}\phi_2(s)
\right\} > q,
\
\frac{1}{m_2(\vartheta_2 - 1)} 
<
\int_{s_0}^{+\infty} \phi_2(\varkappa)\,d{\varkappa},
\ \text{and} \ \mathcal{F}_1 \circ \mathcal{F}_2^{-1}
\!\left(
\frac{1}{m_2(\vartheta_2 - 1)}
\right)
+
\frac{1}{m_1(1-\vartheta_1)}<\int_{s_0}^{+\infty} \phi_1(\varkappa)\,d{\varkappa},
\)
where
\(m_1=
2^{\vartheta_1}
\frac{(1- \aleph)}{\mathcal (1+\aleph)^{1-\lambda_1}}
\), \(
m_2=
2^{\vartheta_2}(1- \aleph)^{\lambda_2}
\) for some $
\aleph\in(0,1)$, and $\vartheta_i = \frac{1+\lambda_i}{2}, \ i=1,2$ with \(
\mathcal{F}_1(s)
=
\int_0^s
\left(
\frac{(1- \aleph)}{\mathcal (1+\aleph)^{1-\lambda_1}} \phi_1(\varkappa) -\delta
\right) d{\varkappa},
\
\mathcal{F}_2(s)
=
\int_0^s
\left(
(1- \aleph)^{\lambda_2} \phi_2(\varkappa) -\delta
\right) d{\varkappa}.
\) Then the trajectory $w(s)$ of \eqref{eqDistr} converges to ${\overline{\mathfrak{w}}}^*$ in $\mathrm{FxT}$.
\end{theorem}

\begin{proof}
Differentiating the function \eqref{THL1} along the model trajectories. Then by Assumption~\eqref{assINDIs} yields
\begin{equation}
    \begin{split}
        \dot{\mathcal{L}}(s,w(s))
&=
\left\langle w-{\overline{\mathfrak{w}}}^*,\dot w\right\rangle
\\
&=
-
\left(
\phi_1(s)\frac{1}{\|w-\Psi(w)\|^{1-\lambda_1}}
+
\phi_2(s)\frac{1}{\|w-\Psi(w)\|^{1-\lambda_2}}
\right)
\cdot
\langle w-{\overline{\mathfrak{w}}}^*, w-\Psi(w)\rangle
\\&+ \delta\|w-{\overline{\mathfrak{w}}}^*\|^2.
    \end{split}
\end{equation}
Using Lemma \ref{tr1}, we get 
\begin{equation}\label{DTHL2}
    \begin{split}
        \dot{\mathcal{L}}(s,w(s))
&\le
-
(1- \aleph) \phi_1(s)
\frac{\|w-{\overline{\mathfrak{w}}}^*\|^2}
{\mathcal (1+\aleph)^{1-\lambda_1}\|w-{\overline{\mathfrak{w}}}^*\|^{1-\lambda_1}}
-
(1- \aleph) \phi_2(s)
\frac{\|w-{\overline{\mathfrak{w}}}^*\|^2}
{(1- \aleph)^{1-\lambda_2}\|w-{\overline{\mathfrak{w}}}^*\|^{1-\lambda_2}}\\
&\qquad+ \delta\|w-{\overline{\mathfrak{w}}}^*\|^2.
\\
&=
-
\frac{(1- \aleph)}{\mathcal (1+\aleph)^{1-\lambda_1}}
\phi_1(s)
\|w-{\overline{\mathfrak{w}}}^*\|^{1+\lambda_1}
-
(1- \aleph)^{\lambda_2} \phi_2(s)
\|w-{\overline{\mathfrak{w}}}^*\|^{1+\lambda_2}\\
&\qquad+ \delta\|w-{\overline{\mathfrak{w}}}^*\|^2.
    \end{split}
\end{equation}
If we suppose
\(
\Xi := \{w \mid \|w-{\overline{\mathfrak{w}}}^*\|\le 1\}.
\) Then it is form two cases.\\

\medskip
\noindent
\textbf{Case 1:} If $w\in \Xi$. Since $0<\lambda_1<1$ with $\|w-{\overline{\mathfrak{w}}}^*\|\le1$, it concluded that
\[
\|w-{\overline{\mathfrak{w}}}^*\|^2 - \|w-{\overline{\mathfrak{w}}}^*\|^{1+\lambda_1} \le 0.
\]
Thus,
\begin{equation}\label{ineq13}
\dot{\mathcal{L}}(s,w(s))
\le
- N_1 \|w-{\overline{\mathfrak{w}}}^*\|^{1+\lambda_1}
- N_2 \|w-{\overline{\mathfrak{w}}}^*\|^{1+\lambda_2},
\end{equation}
where
\(
N_1 =\frac{(1- \aleph) \phi_1(s)}{\mathcal (1+\aleph)^{1-\lambda_1}}-\delta,
\) and \(
N_2=
(1- \aleph)^{\lambda_2} \phi_2(s)-\delta.
\) \\

\noindent
\textbf{Case 2:} If $w \notin \Xi$. Since $\lambda_2>1$ with $\|w-{\overline{\mathfrak{w}}}^*\|>1$, it concluded that
\[
\|w-{\overline{\mathfrak{w}}}^*\|^2 - \|w-{\overline{\mathfrak{w}}}^*\|^{1+\lambda_2} \le 0.
\]
Thus, inequality \eqref{ineq13} is valid for both cases. Consequently, we obtain
\begin{equation*}
\dot{\mathcal{L}}(s,w(s))
\le
-
\left(\frac{(1- \aleph)}{\mathcal (1+\aleph)^{1-\lambda_1}}
\phi_1(s)-\delta\right)
\|w-{\overline{\mathfrak{w}}}^*\|^{1+\lambda_1}
-
\left((1- \aleph)^{\lambda_2}\phi_2(s)-\delta\right)
\|w-{\overline{\mathfrak{w}}}^*\|^{1+\lambda_2}.
\end{equation*}
Equivalently, in terms of L, we have
\[
\dot{\mathcal{L}}(s,w(s))
\le
- 2^{\frac{1+\lambda_1}{2}}
\mathcal{F}_1(s) {\mathcal{L}}^{\frac{1+\lambda_1}{2}}
-
2^{\frac{1+\lambda_2}{2}}
\mathcal{F}_2(s) {\mathcal{L}}^{\frac{1+\lambda_2}{2}}.
\]
where $\mathcal{F}_1=\frac{(1- \aleph)}{\mathcal (1+\aleph)^{1-\lambda_1}}
\phi_1(s)-\delta$ and $\mathcal{F}_2=(1- \aleph)^{\lambda_2}\phi_2(s)-\delta$. Therefore, it follows from Lemma~\ref{lem2} that $w(s)$ converges to 
${\overline{\mathfrak{w}}}^*$ within  $\mathrm{FxT}$.
\end{proof}

\section{Illustrative Examples}\label{ICFxTsec5}
To demonstrate effectiveness of proposed $\mathrm{FxT}$ model~\eqref{INFxTM}, we consider the following IP \eqref{inc1}
\[
\mathcal{P}= 2\mathcal{I}, 
\qquad 
\mathcal{Q}(w) = M^{\top} M w,
\]
where $\mathcal{I}$ with size $n$ and 
$M \in \mathbb{R}^{m \times n}$ is a randomly generated matrix. It is straightforward to verify that the inclusion
\(
0 \in \mathcal{P}(w) + \mathcal{Q}(w)
\)
admits the unique solution ${\overline{\mathfrak{w}}}^* = 0$. Here, we choose
\(
m = 10, 
\ 
n = 8, 
\ 
\mu = 0.004,
\)
and the exponents
\(
\lambda_1 = 0.79,
\
\lambda_2 = 1.2,
\)
which satisfy $0 < \lambda_1 < 1$ and $\lambda_2 > 1$. The time-varying coefficient functions are selected as
\[
\phi_1(s) = 10 \big(1 + 0.2 \sin(4\pi t)\big),
\quad
\phi_2(s) = 190 \big(1 + 0.2 \cos(4\pi t)\big).
\]
Since $\mathcal{P}= 2\mathcal{I}$, the resolvent operator is given by
\(
\mathscr{R}_{\mu \mathcal{P}}(w)
=
(\mathcal{I} + \mu \mathcal{P})^{-1} w
=
\frac{1}{1 + 2\mu} w.
\)
Consequently,
\(
\Psi(w)
=
\mathscr{R}_{\mu \mathcal{P}}(w - \mu \mathcal{Q}(w))
=
\frac{1}{1 + 2\mu}
\big(
w - \mu M^{\top} M w
\big).
\) We simulate the TVCNDM \eqref{INFxTM} is solved using a stiff solver over the interval 
$[0,T_{\max}]$ with $T_{\max} = 4$.

Figure~\ref{figTVInc} illustrates the transient responses of the state variables $w_i(s)$ for the proposed TVCNDM \eqref{INFxTM}. The trajectories of all state components rapidly approach to ${\overline{\mathfrak{w}}}^*$. The logarithmic convergence behavior of the error  $\|w_i(s) - \overline{\mathfrak{w}}_i^*\|$ is shown in Figure~\ref{figTVincErr}, thereby illustrating convergence performance of proposed TVCNDM \eqref{INFxTM}.
\begin{figure}[H]
	\centering
	\begin{subfigure}[b]{0.48\linewidth}
		\centering
		\includegraphics[width=\linewidth]{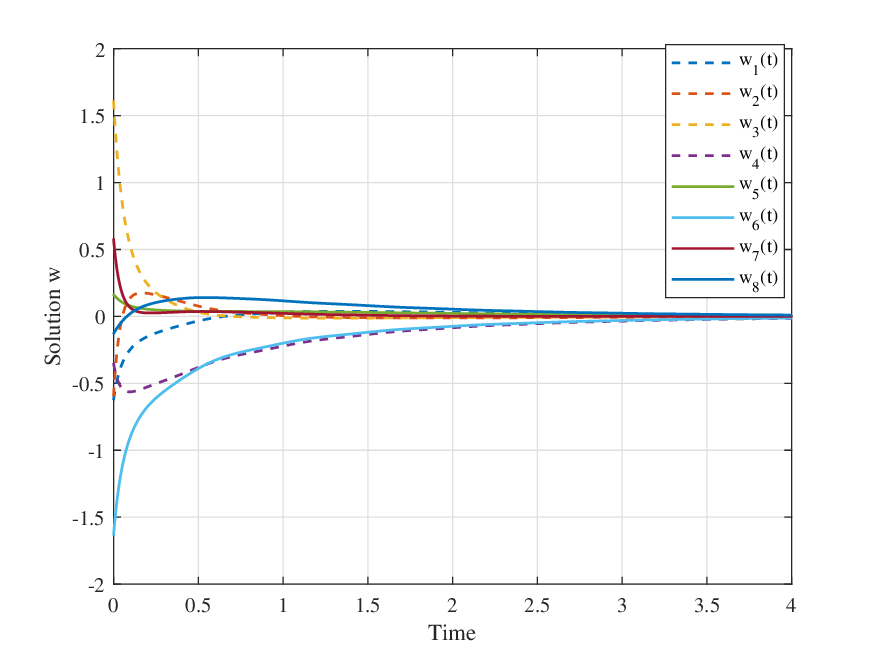}
		\caption{}
		\label{figTVInc}
	\end{subfigure}
	\hfill
	\begin{subfigure}[b]{0.48\linewidth}
		\centering
		\includegraphics[width=\linewidth]{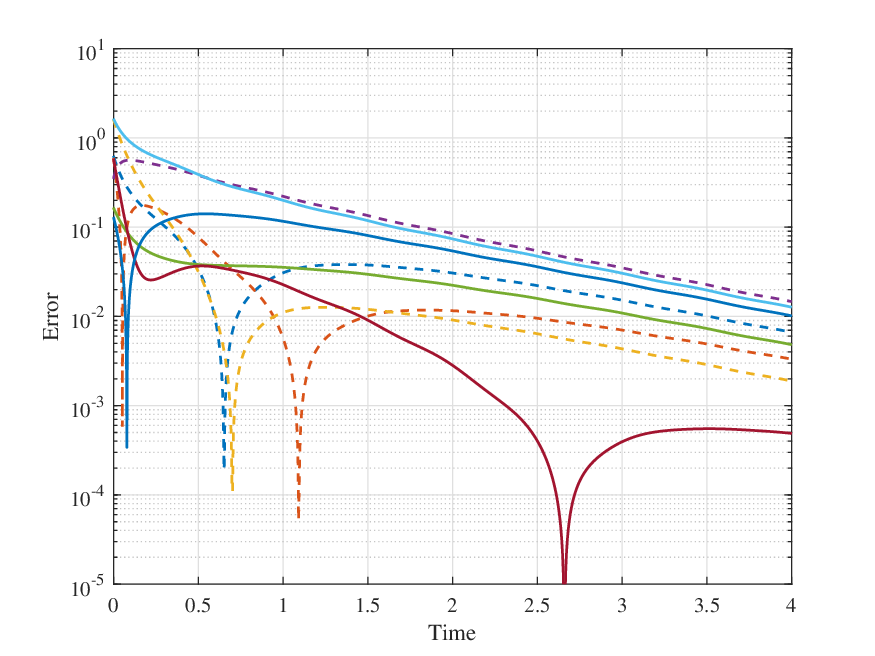}
		\caption{}
		\label{figTVincErr}
	\end{subfigure}
	\caption{(a) State trajectories; (b) Error convergence of TVCNDM \eqref{INFxTM}.}
	\label{figTVcombined}
\end{figure}

\subsection{Illustrate with constant coefficient}
By setting $\phi_1(s)=\varpi_0=15$ and $\phi_2(s)= \varrho_0=110$ in \eqref{INFxTM}, the TVCNDM reduces to the constant-gain dynamical model described in Corollary~\ref{IncCL1}. 
Figures~\ref{figINcoftrj} and~\ref{figINcofErr} illustrate the transient state trajectories and the corresponding logarithmic error convergence of this constant-gain model.
\begin{figure}[H]
	\centering
	\begin{subfigure}[b]{0.48\linewidth}
		\centering
		\includegraphics[width=\linewidth]{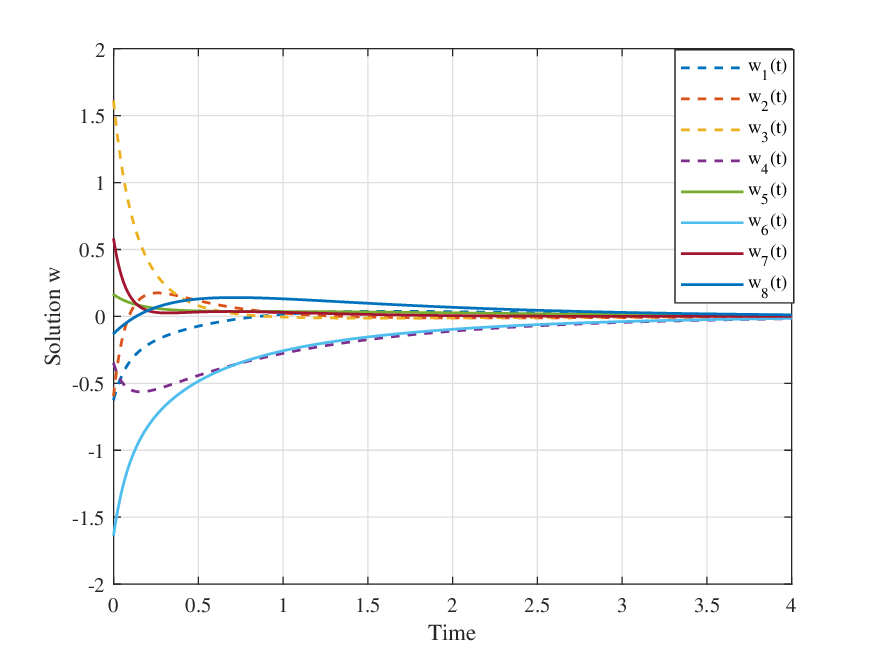}
		\caption{}
		\label{figINcoftrj}
	\end{subfigure}
	\hfill
	\begin{subfigure}[b]{0.48\linewidth}
		\centering
		\includegraphics[width=\linewidth]{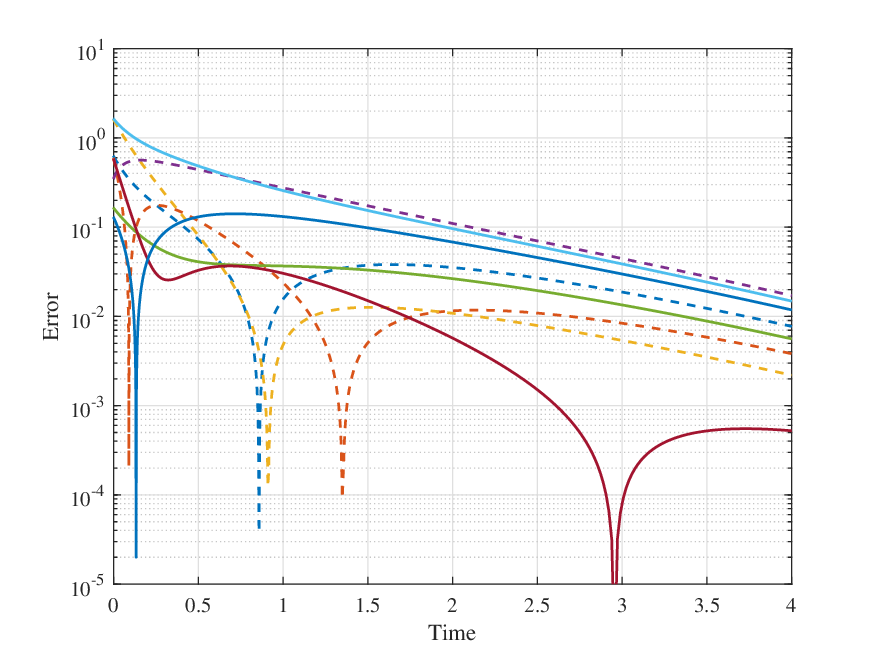}
		\caption{}
		\label{figINcofErr}
	\end{subfigure}
	\caption{(a) State trajectories; (b) Error convergence of TVCNDM \eqref{INFxTM} in the constant-gain case (Corollary \ref{IncCL1} case-1).}
	\label{figINcofcombined}
\end{figure}
\subsection{Error Convergence Analysis}
We consider a comparative study of three neurodynamic models: 
\begin{enumerate}
	\item the $\mathrm{FxT}$ model \eqref{INFxTM} with constant coefficients (Case~1 of Corollary~\ref{IncCL1}),
	\item the $\mathrm{FxT}$ model \eqref{INFxTM} with time-varying coefficients, and
	\item the linear time-varying scaling model \eqref{1stNDS}.
\end{enumerate}
The test problem is the $\ell_1$-regularized quadratic optimization problem
\[
\min_{w \in \mathbb{R}^2} 
\frac{1}{2}\|Aw - b\|^2 + \eta \|w\|_1,
\]
with
\(A= \begin{pmatrix} 3 & 1 \\ 1 & 2 \end{pmatrix},
\
b = \begin{pmatrix} 1 \\ -1 \end{pmatrix},
\
\mu = 0.01,
\
\eta = 0.1.
\)
The initial state is $w_0 = (2,-2)^\top$. The nonlinear exponents are chosen as
\(
0 < \lambda_1 = 0.8 < 1, 
\
\lambda_2 = 1.2 > 1,
\)
ensuring the $\mathrm{FxT}$ structure. For the constant-gain model, we take $\varpi_0=15$ and $\varrho_0=110$.  
For the time-varying case,
\[
\phi_1(s) = 30 + 60(1-e^{-4s}),
\qquad
\phi_2(s) = 30 + 60(1-e^{-4s}),
\]
which are continuous and strictly positive.  
The linear scaling model uses
\(
\sigma(s) = 30 + 60(1-e^{-4s}).
\)
Figure~\ref{figInclusionCom} shows the logarithmic error trajectories 
$\|w(s)-{\overline{\mathfrak{w}}}^*\|$. The $\mathrm{FxT}$ model \eqref{INFxTM} with time-varying coefficient achieves the fastest decay and reaches the prescribed tolerance ($10^{-4}$) in the shortest time.

These results confirm that incorporating continuous time-varying gains significantly enhances transient convergence speed compared to constant-gain and purely linear scaling strategies.
\begin{figure}[H]
	\centering
	\includegraphics[width=0.75\linewidth]{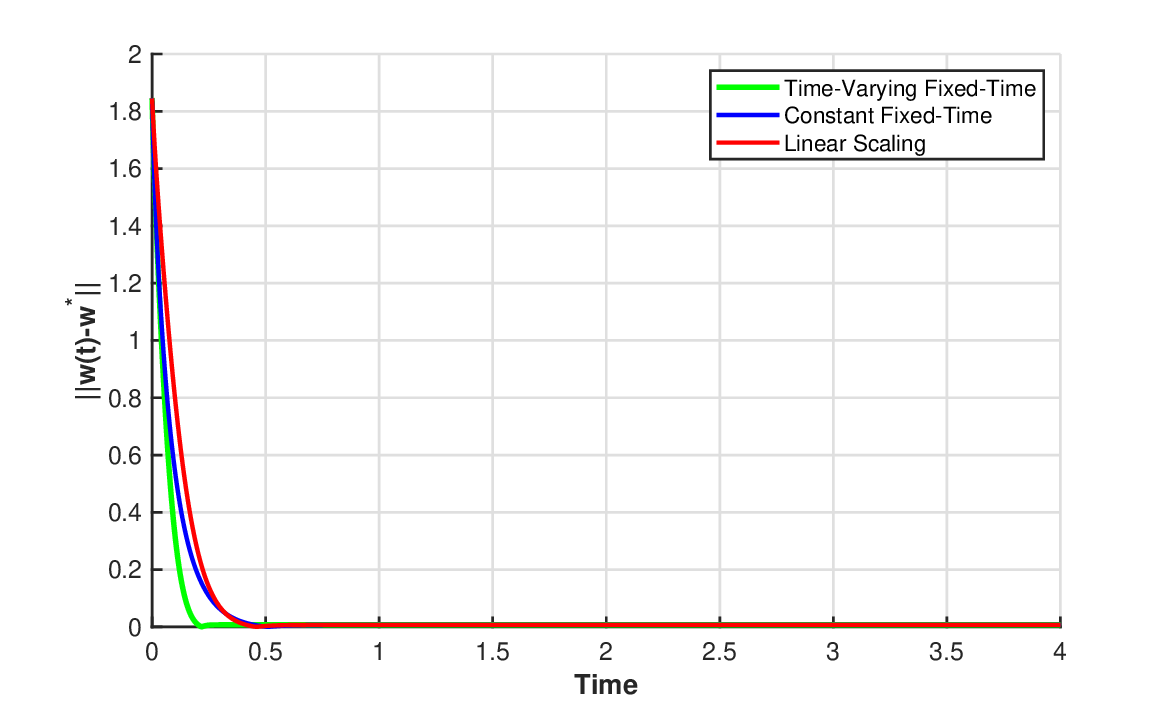}
	\caption{Error convergence comparison with respect to models.}\label{figInclusionCom}
\end{figure}
	
\section{Conclusion}\label{ICFxTsec6}
This paper introduced a TVCNDM for solving the IP \eqref{inc1}. Under a generalized monotonicity framework, existence, uniqueness, and $\mathrm{FxT}$ convergence of model trajectory to solution were established, together with explicit settling-time estimates.

The robustness analysis showed that the proposed TVCNDM maintains $\mathrm{FxT}$ convergence even in the presence of bounded disturbances, ensuring reliable performance. The proposed method was successfully applied numerical simulations validated the theoretical findings, showing fast error convergence and improved transient behavior due to the time-varying design.

Future research will explore the applicability of the proposed framework to stochastic settings, infinite-dimensional spaces, and distributed neurodynamic implementations for large-scale optimization problems.

\section*{CRediT authorship contribution statement}
V. K. Khan: Writing-original draft, writing-review and editing, methodology, investigation, software, conceptualization, visualization, resources; investigation, formal analysis, supervision; Visualization. All authors reviewed and approved the final version of the manuscript prior to publication.

\section*{Conflict of interest}
 No conflicts of interest.

\end{document}